\documentclass{amsart}

\usepackage{amssymb}
\usepackage{amsrefs}
\usepackage{amsmath}
\usepackage{amsfonts}
\usepackage{amsthm}
\usepackage{array}
\usepackage{bm}
\usepackage{enumitem}
\usepackage{xcolor}

\usepackage{eucal}
\usepackage{upgreek}

\usepackage{hyperref}

\usepackage{tikz}
\usetikzlibrary{cd}

\hypersetup{
	colorlinks,
	linkcolor={red!50!black},
	citecolor={blue!50!black},
	urlcolor={blue!80!black}
}

\newtheorem{thm}{Theorem}[section]
\newtheorem{prop}[thm]{Proposition}

\newtheorem{cor}[thm]{Corollary}
\newtheorem*{cor*}{Corollary}
\newtheorem*{thm*}{Theorem}

\theoremstyle{definition}
\newtheorem{definition}[thm]{Definition}

\theoremstyle{plain} % just in case the style had changed
\newcommand{\thistheoremname}{}
\newtheorem{genericthm}[thm]{\thistheoremname}

\newtheorem*{genericthm*}{\thistheoremname}
\newenvironment{namedthm*}[1]
{\renewcommand{\thistheoremname}{#1}%
	\begin{genericthm*}}
	{\end{genericthm*}}

\theoremstyle{remark}

\newtheorem{remark}[thm]{Remark}

\numberwithin{equation}{section}

\newcommand{\R}{\mathbb{R}}
\newcommand{\V}{\mathsf{V}}
\newcommand{\SO}{\mathrm{SO}}
\newcommand{\SU}{\mathrm{SU}}

\newcommand{\G}{\mathrm{G}}
\newcommand{\eps}{\varepsilon}
\newcommand{\ASO}{\mathrm{ASO}}

\begin{document}

\title{Closed $\mathrm{G}_2$-structures with conformally flat metric}

\author{Gavin Ball}
\address{\textsc{Universit\'{e} du Qu\'{e}bec \`{a} Montr\'{e}al}, 	\textsc{D\'{e}partement de math\'{e}matiques}, \textsc{Case postale 8888, succursale centre-ville, Montr\'{e}al (Qu\'{e}bec), H3C 3P8, Canada}}
\email{gavin.ball@cirget.ca}
\urladdr{http://cirget.math.uqam.ca/~gavin/} % Delete if not wanted.

\begin{abstract}
This article classifies closed $\G_2$-structures such that the induced metric is conformally flat. It is shown that any closed $\G_2$-structure with conformally flat metric is locally equivalent to one of three explicit examples. In particular, it follows from the classification that any closed $\G_2$-structure inducing a metric that is both conformally flat and complete must be equivalent to the flat $\G_2$-structure on $\R^7.$
\end{abstract}

\maketitle

\tableofcontents

%%%%%%%%%%%%%%%%%%%%%%%%%%%%%%%%%%%%%%%%%%%%%%%%%%%%%%%%%%
\section{Introduction}
%%%%%%%%%%%%%%%%%%%%%%%%%%%%%%%%%%%%%%%%%%%%%%%%%%%%%%%%%%

A $\G_2$-structure consists of a 7-manifold $M$ endowed with a differential 3-form $\varphi \in \Omega^3 (M)$ that is pointwise equivalent to the 3-form on $\text{Im} \, {\mathbb{O}} \cong \mathbb{R}^7$ constructed from the multiplication table of the octonions. Such a 3-form induces, in a non-linear fashion, a Riemannian metric $g_{\varphi}$ on $M,$ and there is an interesting interplay between the geometry of $\varphi$ and $g_{\varphi}.$ For instance, a $\G_2$-structure $\varphi$ which is both closed and coclosed is called \emph{torsion-free}, and for torsion-free $\G_2$-structures the metric $g_{\varphi}$ is Ricci-flat and has Riemannian holonomy a subgroup of $\G_2$ \cite{BryExcept}. Outside of the torsion-free setting, there are many other interesting examples of $\G_2$-structures for which $\varphi$ satisfies weaker first-order differential conditions.

Given a Riemannian metric $g,$ it is generally a difficult question to decide if there is a $\G_2$-structure $\varphi$ satisfying one of these weaker first-order conditions and such that $g_{\varphi} = g.$ Particular focus has been paid in the literature to the case when $g$ is an Einstein metric \cites{CleyIv07,Bry05,FFM15,ManUg19,FiRaf15}, especially in the homogeneous setting, and many open problems remain in this area.

One particularly interesting class of $\G_2$-structures is the class satisfying the \emph{closed} condition $d \varphi = 0.$ Motivation for the study of closed $\G_2$-structures comes from the construction of compact manifolds with holonomy $\G_2$. All known methods for constructing such manifolds start with a compact manifold $M$ with a closed $\G_2$-structure $\varphi$ such that the 2-form $d^{{*}_{\varphi}} \varphi$ is `small' in some suitable sense, and use a result of Joyce \cite{Joyce96} to perturb $\varphi$ to a torsion-free $\G_2$-structure $\tilde{\varphi}$ on $M.$ Closed $\G_2$-structures are also thought to be the suitable initial structures for which the \emph{Laplacian flow} \cites{Bry05,LotWeiAna,LotWeiLapShi,LotWeiStab} might be expected to produce torsion-free $\G_2$-structures.

The purpose of this article is to classify closed $\G_2$-structures such that the induced metric $g_{\varphi}$ is conformally flat. The result of this classification is given in the following theorem, which is Theorem \ref{thm:Class} of \S\ref{sect:ConfClosedG2}.

\begin{thm*}
	A closed $\G_2$-structure $\varphi$ with (locally) conformally flat induced metric $g_{\varphi}$ is, up to constant rescaling, locally equivalent to one of the three examples on $\R^7, \R^7 \setminus \lbrace 0 \rbrace,$ and $\Lambda^2_+ \mathbb{CP}^2$ presented in \S\ref{ssect:Egs}.
\end{thm*}

Theorem \ref{thm:Class} is a strong rigidity result. There is a vague analogy between closed $\G_2$-structures and so-called \emph{almost K\"ahler structures}, which consist of a symplectic manifold and a compatible metric. However, in contrast to Theorem \ref{thm:Class}, it is known that there are infinitely many almost K\"ahler structures with conformally flat metric \cite{CDDH99}.

Of the closed $\G_2$-structures that feature in the classification, only the flat $\G_2$-structure on $\R^7$ has $g_{\varphi}$ complete, and this has the following consequence, which is Corollary \ref{cor:CompleteCFlt} of \S\ref{sect:ConfClosedG2}.

\begin{cor*}
	Let $(M, \varphi)$ be a 7-manifold endowed with a closed $\G_2$-structure such that $g_\varphi$ is conformally flat and complete. Then $\varphi$ is locally equivalent to the flat $\G_2$-structure $\phi$ on $V = \R^7$ and $M$ is the quotient of $V$ by a discrete group of $\G_2$-automorphisms.
\end{cor*}

Using the method of the moving frame, it is possible to rephrase the problem of finding a closed $\G_2$-structure with conformally flat induced metric in terms of the existence of a coframing on a manifold satisfying a prescribed set of structure equations. Prescribed coframing problems are classical, and are the subject of powerful existence and uniqueness theorems going back to the work of \'Elie Cartan \cites{Cart04,BryEDSNotes}. Theorem \ref{thm:Class} is proven by exhibiting enough examples to apply the uniqueness part of one of these theorems.

Let $\varphi$ be a closed $\G_2$-structure. The 2-form $\tau = d^{{*}_{\varphi}} \varphi$ is called the \emph{torsion 2-form} of $\varphi.$ Let $\lambda \in \R$ be a constant. A closed $\G_2$-structure $\varphi$ is called \emph{$\lambda$-quadratic} \cites{Bry05,Ball19} if it satisfies the equation
\begin{equation}\label{eq:LamQuadint}
d\tau= \tfrac{1}{7}|\tau|^2\varphi + \lambda \left( \tfrac{1}{7} \left\lvert \tau \right\rvert_{\varphi}^2 + {*}_{\varphi} \left(\tau \wedge \tau \right) \right).
\end{equation}
For $\lambda = 1/2,$ the $\lambda$-quadratic equation is equivalent to the Einstein equation for $g_{\varphi},$ and for $\lambda = 1/6$ the $\lambda$-quadratic equation is equivalent to the \emph{extremally Ricci-pinched} condition, which has been studied in several works \cites{Lauret17,LauNic19,Ball19,Bry05,LauNic19b,FiRafLap18}.

As a consequence of the vanishing of a certain $\G_2$-irreducible component of the Weyl tensor (see equation (\ref{eq:RiemCurvDecompWeyl20})), closed $\G_2$-structures with conformally flat metric are $\lambda$-quadratic for $\lambda = -1/8.$ Thus, the examples of \S\ref{ssect:Egs} show that $\lambda$-quadratic $\G_2$-structures exist for $\lambda = -1/8.$ The only other values of $\lambda$ for which non-trivial solutions of (\ref{eq:LamQuadint}) are currently known to exist are $\lambda = -1, 1/6, 1/3, 2/5,$ and $3/4$ \cite{Ball19}. Examples of $\lambda$-quadratic closed $\G_2$-structures with $\lambda = 2/5$ are constructed in the final section of this article via an ansatz motivated by the form of the conformally flat equations, see Theorem \ref{thm:Lam25Class}.

\subsection{Organisation}

Section \ref{sect:GrpG2} contains background information on the group $\G_2$ and its representation theory. 
 
In Section \ref{sect:G2Struct}, the structure equations for closed $\G_2$-structures are developed up to second order, the space of second-order diffeomorphism invariants of a closed $\G_2$-structure is described, and formulas (\ref{eqs:riemcurv}) are derived for the irreducible components of the Riemannian curvature of $g_{\varphi}$ in terms of these invariants.

Section \ref{sect:ConfClosedG2} contains the main results of the article. In \S\ref{ssect:CFlatdiffanal}, equations for conformal flatness are recorded, the structure equations for closed $\G_2$-structures with a conformally flat metric are written down, and some first consequences are derived. In \S\ref{ssect:Egs}, three examples are given of closed $\G_2$-structures satisfying the conformal flatness condition. In \S\ref{ssect:ConfClass}, the existence and uniqueness Theorem \ref{thm:CFlatExist} is stated and proved. Theorem \ref{thm:Class} then follows by combining Theorem \ref{thm:CFlatExist} with a detailed understanding of the examples presented in \S\ref{ssect:Egs}.

In Section \ref{sect:Lam25}, closed $\G_2$-structures solving a set of equations (\ref{eqs:Lam25HCS}) formally very similar to the equations for conformal flatness are studied. In particular, this set of equations implies the $\lambda$-quadratic condition with $\lambda=2/5.$ Due to the formal similarities with the equations for conformal flatness of $g_{\varphi}$, it is possible to classify all closed $\G_2$-structures solving these equations using techniques analogous to those used in \S\ref{sect:ConfClosedG2}.

\subsection{Acknowledgments} The work in this article forms part of my PhD thesis \cite{Ball19}. I thank my advisor Robert Bryant for his encouragement and for many helpful discussions. I would also like to thank the Simons Foundation for funding as a graduate student member of the Simons Collaboration on Special Holonomy in Geometry, Analysis and Physics.

%%%%%%%%%%%%%%%%%%%%%%%%%%%%%%%%%%%%%%%%%%%%%%%%%%%%%%%%%%
\section{The group $\G_2$}\label{sect:GrpG2}
%%%%%%%%%%%%%%%%%%%%%%%%%%%%%%%%%%%%%%%%%%%%%%%%%%%%%%%%%%

This section collects background information on the group $\G_2$ and its representation theory that will be used throughout the rest of the article. The exposition is for the most part a distillation of that of Bryant \cite{Bry05}, with the addition of extra material on representations of $\G_2$ that will be used in later sections.

\subsection{Definition and basic properties}

Let $V = \R^7.$ The group $\mathrm{GL}(V)$ acts on $\Lambda^3 \left( V^* \right)$ with two open orbits, $\Lambda^3_+$ and $\Lambda^3_-$. The stabiliser of an element in either of these orbits is a real form of the complex simple group $\G^{\mathbb{C}}_2$. The stabiliser of an element of $\Lambda^3_+$ is a compact form, while the stabiliser of an element of $\Lambda^3_-$ is a split form. Let $e^1,...,e^7$ denote the canonical basis of $V^{*}$. Using the shorthand notation $e^{ijk}=e^i \wedge e^j \wedge e^k$ for wedge products, the element
$$\phi = e^{123}+e^{145}+e^{167}+e^{246}-e^{257}-e^{347}-e^{356}$$
is an element of $\Lambda^3_+$, the \emph{standard 3-form on V}, and the group $\G_2$ is defined by
$$\G_2 = \{ A \in \mathrm{GL}(V) \mid A^* \phi = \phi \}.$$

The action of $\G_2$ on $V$ preserves the metric
$$g_{\phi}=\left(e^1\right)^2+\left(e^2\right)^2+\left(e^3\right)^2+\left(e^4\right)^2
+\left(e^5\right)^2+\left(e^6\right)^2+\left(e^7\right)^2$$
and volume form
$$\mathrm{vol}_{\phi} = e^1 \wedge e^2 \wedge e^3 \wedge e^4 \wedge e^5 \wedge e^6 \wedge e^7,$$
and it follows that $\G_2$ is a subgroup of $SO(7)$. The Hodge star operator determined by $g_{\phi}$ and $\mathrm{vol}_\phi$ is denoted by $*_\phi.$ Note that $\G_2$ also fixes the 4-form
\begin{equation*}
*_\phi \phi = e^{4567} + e^{2367} + e^{2345} + e^{1357} - e^{1346} - e^{1256} - e^{1247}.
\end{equation*}

\subsubsection{Bryant's $\varepsilon$ symbol}

When working with the group $\G_2$ it is often very convenient to use an $\varepsilon$-notation introduced by Bryant \cite{Bry05}. Let $\varepsilon$ denote the unique symbol that is skew-symmetric in three of four indices and satisfies
\begin{subequations}
	\begin{align}
	\phi &= \tfrac{1}{6} \varepsilon_{ijk} e^{ijk}, \\
	{*}_{\phi}\phi &= \tfrac{1}{24} \varepsilon_{ijkl} e^{ijkl}.
	\end{align}
\end{subequations}

The $\varepsilon$ symbol satisfies the useful identities
\begin{subequations}
	\begin{align}
	\varepsilon_{ijk} \varepsilon_{ijl} &= 6 \delta_{kl}, \\
	\varepsilon_{ijq} \varepsilon_{ijkl} &= 4 \varepsilon_{qkl}, \\
	\varepsilon_{ipq} \varepsilon_{ijk} &= \varepsilon_{pqjk} + \delta_{pj} \delta_{qk} - \delta_{pk} \delta_{qj}, \\
	\varepsilon_{ipq} \varepsilon_{ijkl} &= \delta_{pj} \varepsilon_{qkl} -\delta_{jq} \varepsilon_{pkl} +\delta_{pk} \varepsilon_{jql} -\delta_{kq} \varepsilon_{jpl} +\delta_{pl} \varepsilon_{jkq} -\delta_{lq} \varepsilon_{jkp}.
	\end{align}
\end{subequations}

\subsection{Representation theory of $\G_2$}\label{sect:G2reps}

The group $\G_2$ is a compact simple Lie group of rank two. Thus, each irreducible representation of $\G_2$ is indexed by a pair of integers $(p,q)$ corresponding to the highest weight of the representation with respect to a fixed maximal torus in $\G_2$ endowed with a fixed base for its root system. The irreducible representation associated to $(p,q)$ is denoted by $\mathsf{V}_{p,q}.$ The representations that feature in this work are $\mathsf{V}_{p,0}, \mathsf{V}_{0,q}$ and $\mathsf{V}_{1,1}.$

\subsubsection{The representations $\mathsf{V}_{p,0}$} 

The fundamental representation $\mathsf{V}_{1,0}$ is the standard representation $V = \R^7$ used to define the group $\G_2.$ The group $\G_2$ acts transitively on the unit sphere $S^6 \subset V,$ with isotropy group $\mathrm{SU}(3).$ The representation $\V_{p,0}$ for $p \geq 0$ is isomorphic to $\mathrm{Sym}^p_0(\V_{1,0}).$

\subsubsection{The representations $\V_{0,q}$}\label{sssect:V0qreps}

The other fundamental representation of $\G_2$, $\V_{0,1}$, is isomorphic to the adjoint representation $\mathfrak{g}_2.$ The Lie algebra $\mathfrak{g}_2$ may be defined using the $\varepsilon$ symbol as
\begin{equation}
\mathfrak{g}_2 = \left\lbrace a_{ij} e_{i} \otimes e^j \mid a_{ij}=-a_{ji}, \:\: \varepsilon_{ijk}a_{jk}=0 \right\rbrace.
\end{equation}

A maximal torus for $\mathfrak{g}_2$ is
\begin{equation}
\mathfrak{t} = \left\lbrace \mu_1 E_{23} + \mu_2 E_{45} + \mu_3 E_{67} \mid \mu_1, \mu_2, \mu_3 \in \R, \:\: \mu_1+\mu_2+\mu_3 = 0 \right\rbrace,
\end{equation}
where $E_{ij}$ denotes the element $e_i \otimes e^j - e_j \otimes e^i.$ By Cartan's Theorem on maximal tori, every element of $\mathfrak{g}_2$ is conjugate to an element of $\mathfrak{t}$. The Weyl group of $\G_2$ is the dihedral group $\mathrm{D}_6$ of order 12, and a fundmental Weyl chamber in $\mathfrak{t}$ is given by $0 \leq \mu_1 \leq \mu_2.$

The representations $\V_{0,q}$ for $q \geq 0$ are the irreducible constituents of $\mathrm{Sym}^q(\mathfrak{g}_2)$ of highest weight. In particular, for $q=2,$
\begin{equation}
\V_{0,2} = \left\lbrace s_{ijkl} e_{i} \otimes e^j \otimes e_{k} \otimes e^l \: \vline\ \: \begin{aligned}
& s_{ijkl} = -s_{jikl}, \: s_{ijkl} = s_{klij}, \\
& \varepsilon_{ijk} s_{jklm} =0, \: s_{ijik} = 0, \: s_{ijjk}=0
\end{aligned} \right\rbrace,
\end{equation}
and there is an irreducible decomposition
\begin{equation}
\mathrm{Sym}^2(\mathfrak{g}_2) \cong \V_{0,2} \oplus \V_{2,0} \oplus \V_{0,0}.
\end{equation}

\subsubsection{The representation $\V_{1,1}$} The representation $\V_{1,1}$ has dimension 64. It is the highest weight constituent of the tensor product $\V_{1,0} \otimes \V_{0,1}.$ Explicitly,
\begin{equation}
\V_{1,1} = \left\lbrace c_{ijk} e_i \otimes e_j \otimes e^k \mid \: c_{ijk} = - c_{ikj}, \: \eps_{mjk} c_{ijk} = 0, \: \eps_{mij} c_{ijk} = 0 \right\rbrace,
\end{equation}
and there is a decomposition
\begin{equation}
V \otimes \mathfrak{g}_2 \cong \V_{1,1} \oplus \V_{2,0} \oplus V.
\end{equation}

\subsubsection{$\G_2$-decomposition of exterior forms}

Of fundamental importance to the study of $\G_2$-structures on manifolds is the irreducible decomposition of the exterior powers of the standard representation. This decomposition is given by
\begin{subequations}
	\begin{align}
	\Lambda^2 (V^*) & = \Lambda^2_7  \oplus  \Lambda^2_{14},\\
	\Lambda^3 (V^*) & = \Lambda^3_1  \oplus  \Lambda^3_7 \oplus \Lambda^3_{27},\\
	\Lambda^4 (V^*) & = \Lambda^4_1  \oplus  \Lambda^4_7 \oplus \Lambda^4_{27},\\
	\Lambda^5 (V^*) & = \Lambda^5_7  \oplus  \Lambda^2_{14},
	\end{align}
\end{subequations}
where
\begin{subequations}
\begin{align}
	\Lambda^2_7  &= \left\lbrace *_{\phi} \left( \alpha \wedge *_{\phi} \phi \right) \mid \alpha \in \Lambda^1 (V^*) \right\rbrace \cong V \cong \V_{1,0} \\
	\Lambda^2_{14}  &= \left\lbrace \beta \in \Lambda^1 (V^*) \mid \beta \wedge \phi = 2 *_{\phi} \beta \right\rbrace \cong \mathfrak{g}_2 \cong \V_{0,1}, \\
	\Lambda^3_1  &= \left\lbrace r \phi \mid r \in \mathbb{R} \right\rbrace \cong \mathbb{R} \cong \V_{0,0} \\
	\Lambda^3_7  &= \left\lbrace *_{\phi} \left( \alpha \wedge \phi \right) \mid \alpha \in \Lambda^1 (V^*) \right\rbrace \cong V \cong \V_{1,0}, \\
	\Lambda^3_{27}  &= \left\lbrace \gamma \in \Lambda^3 (V^*) \mid \gamma \wedge \phi = 0, \gamma \wedge *_{\phi} \phi = 0 \right\rbrace \cong \text{Sym}^2_0 (V) \cong \V_{2,0}, 
\end{align}
\end{subequations}
and the Hodge star gives an isomorphism $\Lambda^p_j \cong \Lambda^{7-p}_j.$

%%%%%%%%%%%%%%%%%%%%%%%%%%%%%%%%%%%%%%%%%%%%%%%%%%%%%%%%%%
\section{Structure equations for closed $\G_2$-structures}\label{sect:G2Struct}
%%%%%%%%%%%%%%%%%%%%%%%%%%%%%%%%%%%%%%%%%%%%%%%%%%%%%%%%%%

In this section, the structure equations for closed $\G_2$-structures are derived up to second order. The space of second-order diffeomorphism invariants for closed $\G_2$-structures is described and formulas (\ref{eqs:riemcurv}) are given for the Riemannian curvature of the induced metric in terms of the first and second order invariants of the $\G_2$-structure.

\begin{definition}
	Let $M$ be an oriented $7$-manifold.  A \textit{$\textrm{G}_2$-structure} on $M$ is a differential $3$-form $\varphi \in \Omega^3(M)$ such that $\varphi|_x \in \Lambda^3_+(T_x^*M)$ at each $x \in M$.  That is, at each $x \in M$ there exists a coframe $u \colon T_xM \to V = \mathbb{R}^7$ for which $\varphi|_x = u^*(\phi),$ where $\phi$ is the standard 3-form on $V.$
	
	A $\G_2$-structure $\varphi$ is called \emph{closed} if $d \varphi = 0.$
\end{definition}

A 7-manifold $M$ admits a $\G_2$-structure if and only if it is orientable and spin \cite{Bry05}. It is an interesting open problem, about which next to nothing is known, to determine which orientable, spin 7-manifolds admit closed $\G_2$-structures.

A $\G_2$-structure $\varphi$ defines a $\G_2$-structure on $M$ in the sense of $\G$-structures as follows. Let $\pi: \mathcal{F} \to M$ denote the $V$-coframe bundle of $M$, i.e. the principal right $\textrm{GL}(V)$-bundle over $M$ whose fibre over $x \in M$ consists of $V$-coframes $u: T_x M \xrightarrow{\sim} V.$ The $\G_2$-structure $\varphi$ defines a $\G_2$-subbundle $\mathcal{B} \subset \mathcal{F}$ via
\begin{equation}
\mathcal{B}=\left\lbrace u : T_xM \xrightarrow{\sim} V \mid x \in M, u^* \phi = \varphi_x \right\rbrace.
\end{equation}

Every $\text{G}_2$-structure $\varphi$ on $M$ induces a Riemannian metric $g_\varphi$ and an orientation form $\text{vol}_\varphi$ on $M$ via the inclusion $\G_2 < \SO(7).$ In general, the assignment $\varphi \mapsto g_{\varphi}$ is not injective. However, it is the case that two closed $\G_2$-structures inducing isometric metrics must be equivalent.

On a manifold endowed with a $\G_2$-structure the $\G_2$-type decomposition of the exterior powers $\Lambda^p V^*$ extends to a decomposition of the bundles $\Lambda^p T^* M$ and their spaces of sections $\Omega^p (M).$

\subsection{The first structure equation}

Let $\varphi$ be a closed $\G_2$-structure and let $\mathcal{B}$ denote the induced $\G_2$-coframe bundle. Denote by $\omega$ the $V$-valued tautological 1-form on $\mathcal{B},$ i.e. the $V$-valued 1-form on $\mathcal{B}$ defined by 
\begin{equation}
\omega \left( v \right) = u \left( \pi_{*} \left( v \right) \right) \:\:\:\: \text{for all} \:\:\:\: v \in T_u \mathcal{B}.
\end{equation}
In components, $\omega$ may be written $\omega = \omega_i e^i,$ where $e^i$ is the canonical basis of $V^* = \R^7.$ The form $\omega$ is $\pi$-semibasic and has the reproducing property $\eta^{*} (\omega ) = \eta$ for any local section $\eta$ of $\mathcal{B}.$ Of course, $\omega$ is the pullback of the tautological form on the coframe bundle $\mathcal{F}$ via the inclusion $\mathcal{B} \subset \mathcal{F}.$

The pullbacks of the 3-form $\varphi$ and the 4-form $*_\varphi \varphi$ to $\mathcal{B}$ may be written in terms of the components of $\omega$ as
\begin{align}
\pi^* \varphi &= \tfrac{1}{6} \varepsilon_{ijk} \omega_{i} \wedge \omega_j \wedge \omega_k, \label{eq:phionB} \\
\pi^* {*}_{\varphi}\varphi &= \tfrac{1}{24} \varepsilon_{ijkl} \omega_{i} \wedge \omega_j \wedge \omega_k \wedge \omega_{l}. \label{eq:strphionB}
\end{align}
Similarly, the metric $g_{\varphi}$ and volume form $\textrm{vol}_{\varphi}$ pull back to $\mathcal{B}$ as
\begin{align}
\pi^* g_{\varphi} &= \omega_1^2 + \omega_2^2 + \omega_3^2 + \omega_4^2 + \omega_5^2 + \omega_6^2 + \omega_7^2, \\
\pi^* \textrm{vol}_{\varphi} &= \omega_1 \wedge \omega_2 \wedge \omega_3 \wedge \omega_4 \wedge \omega_5 \wedge \omega_6 \wedge \omega_7.
\end{align}

 Let $\mathcal{F}_{\SO(7)}$ denote the oriented orthonormal coframe bundle of the metric $g_{\varphi}.$ There is an inclusion $\mathcal{B} \subset \mathcal{F}_{\SO(7)}$. By the Fundamental Lemma of Riemannian Geometry, there exists a unique $\mathfrak{so}(7)$-valued 1-form $\psi = \psi_{ij} e_i \otimes e^j,$ the Levi-Civita connection form of $g_{\varphi},$ so that the equation
 \begin{equation}\label{eq:CISO7}
 d \omega_i = - \psi_{ij} \wedge \omega_j
 \end{equation}
 holds on $\mathcal{F}_{\SO(7)}.$
 
Restricted to $\mathcal{B} \subset \mathcal{F}_{\text{SO}(7)}$, the Levi-Civita $1$-form $\psi$ is no longer a connection $1$-form in general.  Indeed, according to the splitting $\mathfrak{so}(7) = \mathfrak{g}_2 \oplus V$, there is a decomposition
\begin{equation}
\psi_{ij} = \theta_{ij} + \eps_{ijk} \gamma_k.
\end{equation}
Here, $\theta = \theta_{ij} e_i \otimes e^j$ is a $\mathfrak{g}_2$-valued connection form on $\mathcal{B}$ (the \emph{natural connection} on $\mathcal{B}$ in the parlance of $\G$-structures), and $\gamma = \gamma_i e_i$ is a $V$-valued $\pi$-semibasic 1-form on $\mathcal{B}.$ Since $\gamma$ is $\pi$-semibasic,
\begin{equation}\label{eq:gamdecomp}
\gamma_i = T_{ij}\omega^j
\end{equation}
for some $\mathrm{End}(V)$-valued function $T = T_{ij} e_i \otimes e^j \colon \mathcal{B} \to \mathrm{End}(V)$.  The $1$-form $\gamma$, and hence the functions $T_{ij}$, encodes the \emph{torsion} of the $\text{G}_2$-structure $\varphi.$ From differentiating equation (\ref{eq:phionB}), closure of $\varphi,$ i.e. the equation $d \varphi = 0$, is equivalent to $T$ taking values in $\mathfrak{g}_2 \subset \mathrm{End}(V),$ i.e. the equation $\varepsilon_{ijk} T_{ij} = 0.$ Substitution of equation (\ref{eq:gamdecomp}) into equation (\ref{eq:CISO7}) gives \emph{Cartan's first structure equation} for closed $\G_2$-structures,
\begin{align}\label{eq:CartanIG2}
d\omega_i = -\theta_{ij} \wedge \omega + \eps_{ijk}T_{kl} \omega_j \wedge \omega_l,
\end{align}
where $T_{ij}$ satisfies $\varepsilon_{ijk} T_{ij} = 0.$ The function $T$ is known as the \emph{torsion tensor} of the closed $\G_2$-structure $\varphi.$ It follows from the general theory of equivalence that the function $T$ is a complete first-order diffeomorphism invariant of $\varphi.$

\subsubsection{The torsion 2-form} The 2-form $\tilde{\tau}$ on $\mathcal{B}$ defined by
\begin{equation}
\tilde{\tau} = 3 \, T_{ij} \omega_i \wedge \omega_j
\end{equation}
is invariant under the $\G_2$-action on $\mathcal{B}.$ It follows that $\tilde{\tau}$ is the pullback to $\mathcal{B}$ of a well-defined 2-form on $M,$ which will be denoted by $\tau.$ This 2-form $\tau$ is an element of $\Omega^2_{14}(M),$ and is called the \emph{torsion 2-form} of the closed $\G_2$-structure $\varphi.$ Differentiation of equation (\ref{eq:strphionB}) using the first structure equation (\ref{eq:CartanIG2}) yields that
\begin{equation}\label{eq:dstrphitau}
d *_{\varphi} \varphi = \tau \wedge \varphi.
\end{equation}
In fact, equation (\ref{eq:dstrphitau}) is often taken to be the definition of $\tau.$

If the function $T,$ or equivalently the 2-form $\tau$, vanishes, then the closed $\G_2$-structure is said to be torsion-free. In this special case, the 3-form $\varphi$ is parallel with respect to the metric $g_{\varphi},$ and the restricted holonomy group of the metric $g_\varphi$ is a subgroup of $\G_2$ \cite{BryExcept}.

\subsection{The second structure equation}

Let $\nabla$ denote the covariant derivative associated to the connection 1-form $\theta$. The exterior derivative of the $\mathfrak{g}_2$-valued function $T$ on $\mathcal{B}$ may be written
\begin{equation}\label{eq:dT}
dT = \nabla T ( \omega ) + \theta \cdot T,
\end{equation}
where $\nabla T$ is a function on $\mathcal{B}$ taking values in $\mathfrak{g}_2 \otimes V^*$, and the notation $\theta \cdot T$ represents the action of the $\mathfrak{g}_2$-valued 1-form $\theta$ on the $\mathfrak{g}_2$-valued function $T$ via the adjoint action of $\mathfrak{g}_2.$

There is an irreducible decomposition 
\begin{equation}
\mathfrak{g}_2 \otimes V \cong \V_{1,1} \oplus \V_{2,0} \oplus V,
\end{equation}
so the function $\nabla T$ may be decomposed as $\nabla T = C + H + X,$ where $C, H,$ and $X$ are functions on $\mathcal{B}$ taking values in $\V_{1,1}, \V_{2,0},$ and $V$ respectively. In components,
\begin{equation}\label{eq:nabT}
\begin{aligned}
\left( \nabla T \right)_{ijk} =& C_{kij} + \eps_{ijl} H_{lk} +3 \eps_{jkl} H_{li} + 3 \eps_{kil}H_{lj} \\
 &+ \eps_{ijkl} X_l + 2 \delta_{jk}X_i - 2 \delta_{ki} X_j,
\end{aligned}
\end{equation}
where $C_{kij}$ and $H_{ij}$ satisfy the symmetries and trace conditions described in \S\ref{sect:G2reps}.

\begin{prop}
	For any closed $\G_2$-structure $\varphi,$ the function $X$ defined by (\ref{eq:nabT}) vanishes identically.
\end{prop}
\begin{proof}
	This is a consequence of the vanishing of the $\Omega^6_7(M)$ component of the equation $d^2 \left( *_{\varphi} \varphi \right).$ One may calculate using equations (\ref{eq:CartanIG2}) and (\ref{eq:dT}) that, on $\mathcal{B},$
	\begin{equation}
	0 = \pi^* {*}_{\varphi} d^2 \left( {*}_{\varphi} \varphi \right) = \pi^* {*}_{\varphi} \left( d \tau \wedge \varphi \right)  = -72 X_i \omega_{i}. \qedhere
	\end{equation}
\end{proof}

Equation (\ref{eq:dT}) may be written in components as
\begin{equation}\label{eq:dTcomp}
dT_{jk}= C_{ijk}\omega_i + \left(\varepsilon_{jkl}H_{li}+ 3 \varepsilon_{kil}H_{lj}+3\varepsilon_{lij}H_{lk}\right)\omega_i + T_{jl}\theta_{lk}-T_{kl}\theta_{lj}.
\end{equation}

\subsubsection{The 3-form $d \tau$} The function $H$ on $\mathcal{B}$ may be written in terms of the components of 3-form $d \tau.$  By equation (\ref{eq:dTcomp}), the pullback of $d \tau$ to $\mathcal{B}$ is
\begin{equation}\label{eq:dtauH}
\pi^* d \tau = \eps_{ikl} \left( 21 H_{ij} - \tfrac{1}{7} T_{im}T_{mj} \right) \omega_{j} \wedge \omega_k \wedge \omega_l.
\end{equation}

\subsubsection{Curvature decomposition and the first Bianchi identity}

The derivative of the $\mathfrak{g}_2$-valued natural connection form $\theta$ takes the form
\begin{equation}
d\theta = -\theta \wedge \theta + K(\omega \wedge \omega),
\end{equation}
or in components,
\begin{equation}\label{eq:dthetcomp}
d\theta_{ij} = -\theta_{ik} \wedge \theta_{kj} + K_{ijkl} \omega_k \wedge \omega_l
\end{equation}
where $K : \mathcal{B} \to \text{Hom}(\Lambda^2(V), \mathfrak{g}_2)$ is the \emph{$\G_2$-curvature function} of $\varphi.$ The space $\mathrm{Hom}(\Lambda^2(V), \mathfrak{g}_2)$ decomposes into $\G_2$-irreducible representations as
\begin{equation}\label{eq:CurvDecomp}
\begin{aligned}
 \text{Hom}(\Lambda^2(V), \mathfrak{g}_2) &\cong \left(\V_{0,1} \oplus \V_{1,0} \right)  \otimes \V_{0,1} \\
& \cong \left( \V_{0,2} \oplus \V_{2,0} \oplus \V_{0,1} \oplus \V_{0,0} \right) \oplus \left(\V_{1,1} \oplus \V_{2,0} \oplus \V_{1,0} \right),
\end{aligned}
\end{equation}
so the function $K$ may decomposed analogously as
\begin{equation}
K_{ijkl} = S_{ijkl} + J_{ijkl} + U_{ijkl} + r_{ijkl} + \eps_{mkl} \left( B_{ijm} + L_{ijm} + W_{ijm} \right),
\end{equation}
where
\begin{subequations}
	\begin{align}
	J_{ijkl} =& \left( 2 \eps_{pij} \eps_{qkl} -3 \eps_{pik} \eps_{qjl} -3 \eps_{pil} \eps_{qjk} \right) J_{pq} \\
	&+ 3 \left( \delta_{ik} J_{jl} - \delta_{il} J_{jk} - \delta_{jk} J_{il} + \delta_{jl} J_{ik} \right), \nonumber \\
	U_{ijkl} =& \eps_{pij} \eps_{qkl} U_{pq} - 3 \left( \delta_{ik} U_{jl} - \delta_{il} U_{jk} - \delta_{jk} U_{il} + \delta_{jl} U_{ik} \right), \\
	r_{ijkl} =& \left( \eps_{ijkl} - 2 \delta_{ik} \delta_{jl} + 2 \delta_{jk} \delta_{il} \right) r, \\
	L_{ijm} =& \eps_{pij} L_{pm} + 3 \eps_{pjm} L _{pi} + 3 \eps_{pmi} L_{pj}, \\
	W_{ijm} =& \eps_{ijml} W_l -2 \delta_{im} W_j + 2 \delta_{jm} W_i,
	\end{align}
\end{subequations}
and $S = S_{ijkl} e_i \otimes e^j \otimes e_k \otimes e^l,$ $J = J_{ij} e_i e_j,$ $U = U_{ij} e_i \otimes e^j,$ $r,$ $B = B_{ijk} e_i \otimes e_j \otimes e^k,$ $L = L_{ij} e_i e_j,$ and $W = W_i e_i$ are functions on $\mathcal{B}$ taking values in $\V_{0,2},$ $\V_{2,0},$ $\V_{0,1},$ $\R,$ $\V_{1,1},$ $\V_{2,0},$ and $V$ respectively, so that these symbols satisfy the symmetries and trace conditions described in \S\ref{sect:G2reps}.

It is known that the $\G_2$ curvature function of a torsion-free $\G_2$-structure takes values in the $\V_{0,2}$ summand of (\ref{eq:CurvDecomp}), so by abstract principles the other summands may be expressed in terms of $T$ and $\nabla T.$ This can be done concretely by solving the first Bianchi identity, i.e. the equation $d^2 \omega = 0$. The result is
\begin{equation}\label{eq:BianchI}
\begin{aligned}
U&=0, \: W=0, \: Q=0, \\
B &= C, \\
r &= \tfrac{1}{14} T_{ij}T_{ij}, \\
J_{ij} &= - \tfrac{7}{12} H_{ij} - \tfrac{5}{24} \left( T_{ik}T_{kj} + \tfrac{1}{7} \delta_{ij} T_{kl}T_{kl} \right), \\
L_{ij} &= - \tfrac{4}{3} H_{ij}  - \tfrac{1}{3} \left( T_{ik}T_{kj} + \tfrac{1}{7} \delta_{ij} T_{kl}T_{kl} \right).
\end{aligned}
\end{equation}
It follows from the general theory that the quantities $C, H,$ and $S$ form a complete set of second order diffeomorphism invariants for a closed $\G_2$-structure.

\subsubsection{Curvature of $g_{\varphi}$}

The Levi-Civita form $\psi$ of the metric $g_\varphi$ is given in components by $\psi_{ij} = \theta_{ij} + \varepsilon_{ijk}T_{kl}\omega_l.$ With (\ref{eq:BianchI}) in hand, it is possible to write the Riemann curvature tensor, $\text{Riem}({g_\varphi}),$ of $g_{\varphi}$ in terms of the invariants $T, C, H,$ and $S$. The Riemann curvature tensor decomposes as an $\SO(7)$ representation into the scalar, traceless Ricci, and Weyl tensors,
\begin{equation}
\text{Riem}(g_{\varphi}) = \text{Scal}(g_{\varphi}) + \text{Ric}^0(g_\varphi) + \text{Weyl}(g_{\varphi}).
\end{equation}
In terms of $\G_2$-representations, $\mathrm{Scal}(g_{\varphi})$ is $\V_{0,0}$-valued, $\mathrm{Ric}^0(g_\varphi)$ is $\V_{2,0}$-valued, and $\mathrm{Weyl}(g_{\varphi})$ is $\V_{2,0} \oplus \V_{1,1} \oplus \V_{0,2}$-valued. Formulas for the $\G_2$-components of $\text{Riem}(g_{\varphi})$ are found by substituting equations (\ref{eq:dT}) and (\ref{eq:dthetcomp}) into the Riemannian second structure equation
\begin{equation}
d \psi_{ij} + \psi_{ik} \wedge \psi_{kj} = \tfrac{1}{2} \mathrm{Riem}_{ijkl} \omega_k \wedge \omega_l,
\end{equation}
and decomposing into $\G_2$-irreducible pieces. The result is
\begin{subequations}\label{eqs:riemcurv}
	\begin{align}
	\text{Scal}(g_{\varphi}) &= - 9 T_{ij}T_{ij} = - \tfrac{1}{2} \left\lvert \tau \right\rvert_{\varphi}^2 \label{eq:RiemCurvDecompScal} \\
	\text{Ric}^0(g_{\varphi})_{ij} &= -42H_{ij} - 12 \left( T_{ik}T_{kj} + \tfrac{1}{7} \delta_{ij} T_{kl}T_{kl} \right) \label{eq:RiemCurvDecompRic0} \\
	\left[\text{Weyl}(g_{\varphi})\right]^{2,0}_{ij} &= 4H_{ij} + \left( T_{ik}T_{kj} + \tfrac{1}{7} \delta_{ij} T_{kl}T_{kl} \right) \label{eq:RiemCurvDecompWeyl20} \\
	\left[\text{Weyl}(g_{\varphi})\right]^{1,1}_{ijk} &= C_{ijk} \label{eq:RiemCurvDecompWeyl11} \\
	\left[\text{Weyl}(g_{\varphi})\right]^{0,2}_{ijkl} &= S_{ijkl} - 3 \left( \tfrac{1}{16} \left( 2\eps_{pij}\eps_{qkl} - 3 \eps_{pik}\eps_{qjl} + 3 \eps_{pil}\eps_{qjk} \right) T_{pm}T_{mq} \right. \label{eq:RiemCurvDecompWeyl02} \\
	& \left. + \tfrac{3}{16} \left( \delta_{ik} T_{jm}T_{ml} - \delta_{il} T_{jm}T_{mk} - \delta_{jk} T_{im}T_{ml} + \delta_{jl} T_{im}T_{mk} \right) \right. \nonumber \\
	& \left. \tfrac{1}{12} T_{pq}T_{pq} \left( \eps_{ijkl} + \tfrac{1}{4} \delta_{ik} \delta_{jl} \right)  \right). \nonumber
	\end{align}
\end{subequations}

Expressions equivalent to these have appeared in the literature before. Formulas for the scalar and traceless Ricci curvatures have been found by Bryant \cite{Bry05}, and formulas for the Weyl curvature have been found by Cleyton and Ivanov \cite{CleyIv08}.

%%%%%%%%%%%%%%%%%%%%%%%%%%%%%%%%%%%%%%%%%%%%%%%%%%%%%%%%%%
\section{Closed $\G_2$-structures with conformally flat metric}\label{sect:ConfClosedG2}
%%%%%%%%%%%%%%%%%%%%%%%%%%%%%%%%%%%%%%%%%%%%%%%%%%%%%%%%%%

\subsection{The differential analysis}\label{ssect:CFlatdiffanal}

Let $M$ be a 7-manifold endowed with a closed $\G_2$-structure $\varphi,$ and suppose that the metric $g_{\varphi}$ is (locally) conformally flat. In dimensions greater than 3, local conformal flatness is equivalent to the vanishing of the Weyl curvature. Thus, from equations (\ref{eq:RiemCurvDecompWeyl20}-\ref{eq:RiemCurvDecompWeyl02}), the second order invariants $H, C,$ and $S$ of $\varphi$ must satisfy
\begin{subequations}\label{eqs:HCSConf}
	\begin{align}
	H_{ij} &= \tfrac{1}{4} T_{ik}T_{kj} + \tfrac{1}{28} \delta_{ij} T_{pq}T_{pq}, \label{eq:CFlatH} \\
	C_{ijk} &= 0, \label{eq:CFlatC} \\
	S_{ijkl} &= 3 T_{ij}T_{kl} +\tfrac{3}{16} \left( 2\eps_{pij}\eps_{qkl} - 3 \eps_{pik}\eps_{qjl} + 3 \eps_{pil}\eps_{qjk} \right) T_{pm}T_{mq} \label{eq:CFlatS} \\
	& + \tfrac{9}{16} \left( \delta_{ik} T_{jm}T_{ml} - \delta_{il} T_{jm}T_{mk} - \delta_{jk} T_{im}T_{ml} + \delta_{jl} T_{im}T_{mk} \right) \nonumber \\
	& + \tfrac{1}{16}  \left( 4 \eps_{ijkl} + \delta_{ik} \delta_{jl} - \delta_{il} \delta_{jk} \right) T_{pq}T_{pq}. \nonumber
	\end{align}
\end{subequations}

Substituting equations (\ref{eqs:HCSConf}) into the structure equations (\ref{eq:CartanIG2}), (\ref{eq:dTcomp}), and (\ref{eq:dthetcomp}) gives the equations
\begin{subequations}\label{eqs:CFlatStruct}
	\begin{align}
	d & \omega_i = -\theta_{ij} \wedge \omega_j + \eps_{ijk} T_{kl} \omega_j \wedge \omega_l, \label{eq:CFlatCartI} \\
	d & T_{jk} = \tfrac{1}{4} \left( \eps_{jkl} T_{lm}T_{mi} + 3 \eps_{kil} T_{lm}T_{mj} + 3 \eps_{ijl}T_{lm}T_{mk} + \eps_{ijk}T_{pq}T_{pq} \right)\omega_i \label{eq:CFlatdT} \\
	& + T_{jl}\theta_{lk}-T_{kl}\theta_{lj}, \nonumber \\
	d & \theta_{ij} + \theta_{ik} \wedge \theta_{kj} = \left( T_{ij}T_{kl} +\tfrac{1}{2} \left( 2\eps_{pij}\eps_{qkl} - 3 \eps_{pik}\eps_{qjl} + 3 \eps_{pil}\eps_{qjk} \right) T_{pm}T_{mq} \right. \label{eq:CFlatCartII} \\
	& - \left. \tfrac{5}{2} \left( \delta_{ik} T_{jm}T_{ml} - \delta_{il} T_{jm}T_{mk} - \delta_{jk} T_{im}T_{ml} + \delta_{jl} T_{im}T_{mk} \right) \right. \nonumber \\
	&+  \left. 2 \left( \eps_{ijkq} T_{lm}T_{mq} - \eps_{lijq}T_{km}T_{mq} \right) \right. \nonumber \\
	& \left. + \left( \tfrac{5}{4} \eps_{ijkl} - \delta_{ik} \delta_{jl} + \delta_{il} \delta_{jk} \right) T_{pq}T_{pq} \right) \omega_k \wedge \omega_l. \nonumber
	\end{align}
\end{subequations}

Equations (\ref{eqs:CFlatStruct}) will be collectively referred to as the \emph{structure equations} of a closed $\G_2$-structure with conformally flat metric. The exterior derivatives of (\ref{eqs:CFlatStruct}) are identities.

\subsubsection{First consequences} 

Equations (\ref{eq:CFlatH}) and (\ref{eq:dtauH}) taken together imply that for a closed $\G_2$-structure $\varphi$ with $g_{\varphi}$ conformally flat, the 3-form $d \tau$ satisfies
\begin{equation}
d \tau = \tfrac{1}{8} \left\lvert \tau \right\rvert^2_{\varphi} \varphi - \tfrac{1}{8} {*}_{\varphi} \left( \tau \wedge \tau \right).
\end{equation}
A closed $\G_2$-structure is said to be \emph{$\lambda$-quadratic} \cites{Bry05,Ball19} if $\varphi$ satisfies
\begin{equation}\label{eq:LamQuad}
d\tau= \tfrac{1}{7}|\tau|^2\varphi + \lambda \left( \tfrac{1}{7} \left\lvert \tau \right\rvert_{\varphi}^2 + {*}_{\varphi} \left(\tau \wedge \tau \right) \right),
\end{equation}
for some constant $\lambda \in \R.$ Thus, a closed $\G_2$-structure $\varphi$ with $g_{\varphi}$ conformally flat is $\lambda$-quadratic for $\lambda = -1/8.$

\begin{prop}\label{prop:CpctCFlat}
	If $(M, \varphi)$ is a compact 7-manifold endowed with a closed $\G_2$-structure such that $g_{\varphi}$ is (locally) conformally flat, then the metric $g_{\varphi}$ is flat and the $\G_2$-structure is locally equivalent to the standard $\G_2$-structure $\phi$ on $V = \R^7$
\end{prop}

\begin{proof}
	Bryant \cite{Bry05} has shown that on a compact manifold the only possible $\lambda$-quadratic closed $\G_2$-structures with $\tau$ not identically zero have $\lambda = 1/6.$ Thus, $\tau$ is identically zero and the structure equations (\ref{eqs:CFlatStruct}) reduce to
	\begin{subequations}
		\begin{align}
		d \omega_i &= -\theta_{ij} \wedge \omega_j, \\
		d \theta_{ij} &= - \theta_{ik} \wedge \theta_{kj},
		\end{align}
	\end{subequations}
	the structure equations of the Lie group $\G_2 \ltimes V.$ It follows that $\varphi$ is locally equivalent to $\phi.$
\end{proof}

\begin{remark}
	If $\varphi$ is a closed $\G_2$-structure, then a $\G_2$-structure of the form $e^f \varphi$ is not closed unless $f$ is a constant. In particular, a closed $\G_2$-structure inducing a conformally flat metric is not generally conformally torsion-free or flat itself.
\end{remark}

\subsection{Examples}\label{ssect:Egs}

This section contains three examples of closed $\G_2$-structures with conformally flat induced metric, on the manifolds $\R^7, \R^7 \setminus \lbrace 0 \rbrace,$ and $\Lambda^2_+ \mathbb{CP}^2.$

\subsubsection{The flat $\G_2$-structure on $\R^7$}\label{sssect:EgFlat} The first example is unremarkable. The standard 3-form $\phi$ on $V = \R^7$ is closed, and the induced metric $g_\phi$ is just the standard flat metric on $\R^7,$ hence a fortiori conformally flat.

It follows from the proof of Proposition \ref{prop:CpctCFlat} that any torsion-free $\G_2$-structure $\varphi$ with $g_{\varphi}$ conformally flat is equivalent to the standard $\G_2$-structure on $\R^7.$

\subsubsection{A non-flat example on $\R^7 \setminus \lbrace 0 \rbrace$}\label{sssect:EgNonFlatR7} Let $\mathrm{ASO}(4)$ denote the special Euclidean group in four dimensions, i.e. the group $\SO(4) \ltimes \R^4$ of rigid motions of $\R^4.$ The group $\mathrm{ASO}(4)$ can be regarded as the matrix group consisting of the $5 \times 5$ matrices of the form
\begin{equation}
g = \left( \begin{array}{cc}
1 & 0 \\
v & A
\end{array} \right),
\end{equation}
where $v \in \R^4$ is a column vector and $A \in \SO(4).$ Write the canonical left-invariant Maurer-Cartan form on $\ASO(4)$ as
\begin{equation}
\mu = \left( \begin{array}{ccccc}
0 & 0 & 0 & 0 & 0 \\
\eta_1 & 0 & \xi_1 + \rho_1 & \xi_2 + \alpha_2 & \xi_3 + \alpha_1 \\
\eta_2 & -\xi_1-\rho_1 & 0 & -\xi_3 + \alpha_1 & \xi_2 - \alpha_2 \\
\eta_3 & -\xi_2 - \alpha_2 & \xi_3 - \alpha_1 & 0 & -\xi_1 + \rho_1 \\
\eta_4 & -\xi_3-\alpha_1 & -\xi_2 + \alpha_2 & \xi_1-\rho_1 & 0
\end{array} \right).
\end{equation}
The subgroup generated by the dual vectors to $\xi_1, \xi_2, \xi_3$ and $\rho_1$ is $\mathrm{U}(2) \subset \SO(4).$ 

On the manifold $\R_+ \times \ASO(4),$ with coordinate $t$ in the $\R_+$-direction, define a $\ASO(4)$-invariant 3-form $\tilde{\varphi}$ by
\begin{equation}
\begin{aligned}
\tilde{\varphi} =& \tfrac{4}{9} t^2 d t \wedge \alpha_1 \wedge \alpha_2 + \tfrac{256}{t^4} d t \wedge \left( \eta_1 \wedge \eta_2 + \eta_3 \wedge \eta_4 \right) \\
&+ \tfrac{512}{3t^3} \left( \alpha_1 \wedge \left( \eta_1 \wedge \eta_3 - \eta_2 \wedge \eta_4 \right) + \alpha_2 \wedge \left( -\eta_1 \wedge \eta_4 - \eta_2 \wedge \eta_3 \right) \right).
\end{aligned}
\end{equation}
The Maurer-Cartan equation $d \mu = -\mu \wedge \mu$ implies that $\tilde{\varphi}$ satisfies $d \tilde{\varphi} = 0.$ In fact, $\tilde{\varphi}$ is the pullback to $\R_+ \times \ASO(4)$ of a closed $\G_2$-structure $\varphi$ on the space $\R_+ \times \ASO(4)/\mathrm{U}(2).$ Letting $\pi : \ASO(4) \to \ASO(4)/\mathrm{U}(2)$ denote the coset projection, the metric $g_{\varphi}$ satisfies
\begin{equation}
\pi^* g_{\varphi} = dt^2 + \tfrac{4}{9} t^2 \left( \alpha_1^2 + \alpha_2^2 \right) + \tfrac{256}{t^4} \left( \eta_1^2 + \eta_2^2 + \eta_3^2 + \eta_4^2 \right).
\end{equation} 
Computation using the Maurer-Cartan equation shows that $g_{\varphi}$ is conformally flat. In fact, the metric $t^4 g_{\varphi}$ is isometric to $\R^7$ with the flat metric.

The torsion 2-form $\tau$ of $\varphi$ satisfies
\begin{equation}
\pi^* \tau = \tfrac{32}{9} t \alpha_1 \wedge \alpha_2 - \tfrac{1024}{t^5} \left( \eta_1 \wedge \eta_2 + \eta_3 \wedge \eta_4 \right),
\end{equation}
and the scalar curvature of $g_{\varphi}$ may be computed from the formula (\ref{eq:RiemCurvDecompScal}) to be
\begin{equation}
\mathrm{Scal}(g_{\varphi}) = -\frac{48}{t^2}.
\end{equation}
It follows that the metric $g_{\varphi}$ is incomplete.

The space $\ASO(4)/\mathrm{U}(2)$ is the unit sphere bundle of the bundle of self-dual 2-forms on $\R^4,$ so it is diffeomorphic to $\R^4 \times S^2.$ The integral manifolds of the distribution defined by $\eta_1 = \eta_2 = \eta_3 = \eta_4 = 0$ are the submanifolds of the form $\left\lbrace p \right\rbrace \times S^2$ for $p \in \R^4,$ while the integral manifolds of the distribution $\alpha_1 = \alpha_2 = 0$ are the submanifolds of the form $\R^4 \times \left\lbrace q \right\rbrace$ for $q \in S^2.$ Thus, the manifold $\R_+ \times \ASO(4)/\mathrm{U}(2)$ is diffeomorphic to $\R^7 \setminus \left\lbrace 0 \right\rbrace.$

\subsubsection{An example on $\Lambda^2_+ \mathbb{CP}^2$}\label{sssect:Lam2CP2}

Write the left-invariant Maurer-Cartan form of $\SU(3)$ as
\begin{equation}
\mu = \left( \begin{array}{ccc}
\tfrac{2}{3} i \xi & \eta_4 + i \eta_5 & -\eta_7 - i \eta_6 \\
-\eta_4 + i \eta_5 & i \zeta - \tfrac{1}{3} i \xi & \eta_2 + i \eta_3 \\
\eta_7 - i \eta_6 & -\eta_2 + i \eta_3 & - i \zeta - \tfrac{1}{3} i \xi
\end{array} \right).
\end{equation}
The subgroup generated by the dual vectors to $\xi$ and $\zeta$ is a maximal torus $\mathrm{T}^2 \subset \SU(3),$ while the subgroup generated by the vectors dual to $\xi, \zeta, \eta_2$ and $\eta_3,$ is isomorphic to $\mathrm{U}(2)$ and will be denoted  $\mathrm{U}(2)_1.$

On the manifold $\R_{+} \times \SU(3),$ with coordinate $r$ in the $\R^{+}$-direction, define 1-forms $\alpha_1, \ldots, \alpha_7$ by
\begin{equation}
\begin{aligned}
\alpha_1 &= \frac{2 \: d r}{\cosh(r)^{\frac{2}{3}}} , & \left( \begin{array}{c}
\alpha_2 \\
\alpha_3
\end{array} \right) &= 4 \sinh(r) \cosh(r)^{\frac{1}{3}} \left( \begin{array}{c}
\eta_2 \\
\eta_3
\end{array}\right), \\
\left( \begin{array}{c}
\alpha_4 \\
\alpha_5
\end{array} \right) & = \frac{2+e^{2r}}{\cosh(r)^{\frac{2}{3}}} \left( \begin{array}{c}
\eta_4 \\
\eta_5
\end{array}\right), & \left( \begin{array}{c}
\alpha_6 \\
\alpha_7
\end{array} \right) & = \frac{2+e^{-2r}}{\cosh(r)^{\frac{2}{3}}} \left( \begin{array}{c}
\eta_6 \\
\eta_7
\end{array}\right),
\end{aligned}
\end{equation}
and a 3-form $\tilde{\varphi}$ by
\begin{equation}
\tilde{\varphi} =  \alpha_{123} + \alpha_{145} + \alpha_{167} + \alpha_{246} - \alpha_{257} - \alpha_{347} - \alpha_{356},
\end{equation}
where $\alpha_{ijk}$ stands for the wedge product $\alpha_i \wedge \alpha_j \wedge \alpha_k.$ The Maurer-Cartan equation $d \mu = - \mu \wedge \mu$ implies that $\tilde{\varphi}$ satisfies $d \tilde{\varphi} = 0.$ In fact, $\tilde{\varphi}$ is the pullback to $\R_+ \times \SU(3)$ of a closed $\G_2$-structure $\varphi$ on the space $\R_{+} \times \SU(3)/\mathrm{T}^2.$ The $\G_2$-structure $\varphi$ is cohomogeneity-one under the action of $\SU(3).$ It is not hard to check that $\varphi$ satisfies the conditions \cite{CleSwa} required to extend $\SU(3)$-equivalently over $\left\lbrace r = 0 \right\rbrace$ and defines a closed $\G_2$-structure on the manifold $\Lambda^2_+ \mathbb{CP}^2.$ The singular orbit $\mathbb{CP}^2 = \SU(3)/\mathrm{U}(2)_1$ of the $\SU(3)$-action on $\Lambda^2_+ \mathbb{CP}^2$ is exactly the locus $\left\lbrace r = 0 \right\rbrace.$

Letting $\pi : \SU(3) \to \SU(3)/\mathrm{T}^2$ denote the coset projection, the metric $g_{\varphi}$ satisfies
\begin{equation}
\pi^* g_{\varphi} = \alpha_1^2 + \alpha_2^2 + \alpha_3^2 + \alpha_4^2 + \alpha_5^2 + \alpha_6^2 + \alpha_7^2.
\end{equation}
Computation using the Maurer-Cartan equation shows that $g_{\varphi}$ is conformally flat. In fact, the metric
\begin{equation}
h = \frac{3 \cosh(r)^{\frac{4}{3}}}{\left( 4 \cosh(r)^2 -1 \right)^2} g_{\varphi}
\end{equation}
is isometric to the round metric on the 7-sphere. It is known \cite{Miy06} that $S^7 \setminus \mathbb{CP}^2 \cong \Lambda^2_{+} \mathbb{CP}^2,$ and the metric $h$ is simply the restriction of the round metric on $S^7$ to this subset.

The torsion 2-form $\tau$ of $\varphi$ satisfies
\begin{equation}
\tau = \frac{4 \sinh(r)}{3 \cosh(r)^{\frac{1}{3}}} \alpha_{23} + \frac{6 \cosh(r) - 2 \sinh(r)}{3 \cosh(r)^{\frac{1}{3}}} \alpha_{45} - \frac{6 \cosh(r) + 2 \sinh(r)}{3 \cosh(r)^{\frac{1}{3}}} \alpha_{67},
\end{equation}
and the scalar curvature of $g_{\varphi}$ may be computed using formula (\ref{eq:RiemCurvDecompScal}) to be
\begin{equation}
\mathrm{Scal}(g_{\varphi}) = \frac{4 - 16 \cosh(r)^2}{3 \cosh(r)^{\frac{2}{3}}} .
\end{equation}
Note that the scalar curvature blows up as $r \to \infty,$ and this is at finite distance, since
\begin{equation}
\int_{r=0}^{r=\infty} \alpha_1 = \int_{r=0}^{r=\infty} \frac{4 \: d r}{\cosh(r)^{\frac{2}{3}}} < \infty.
\end{equation}
It follows that the metric $g_{\varphi}$ is incomplete.

\subsection{The classification theorem}\label{ssect:ConfClass}

This section contains the statement and proof of the classification of closed $\G_2$-structures with conformally flat induced metric.

The following theorem is a consequence of \'Elie Cartan's generalisation \cite{Cart04} of Lie's Third Fundamental Theorem. A modern exposition of this and related theorems concerning prescribed coframing problems has been given by Bryant \cite{BryEDSNotes}.

\begin{thm}\label{thm:CFlatExist}
	For any $T_0 \in \mathfrak{g}_2,$ there exists a closed $\G_2$-structure $\varphi$ with conformally flat metric $g_{\varphi}$ on a neighbourhood $U$ of $0 \in \R^7$ whose $\G_2$-coframe bundle $\pi : \mathcal{B} \to U$ contains a $u_0 \in \mathcal{B}_0 = \pi^{-1} \left( 0 \right)$ for which $T \left( u_0 \right) = T_0.$ Any two real analytic closed $\G_2$-structures with conformally flat metric that satisfy this property are isomorphic on a neighbourhood of $0 \in \R^7.$ Furthermore, any closed $\G_2$-structure with conformally flat metric that is $C^2$ is real-analytic.
\end{thm}

\begin{proof}
	This is an application of Theorem 2 of \cite{BryEDSNotes}. Since the exterior derivatives of equations (\ref{eqs:CFlatStruct}) are identities, the hypotheses of Cartan's Theorem are satisfied. Thus, for any $T_0 \in \mathfrak{g}_2$ there exists a real-analytic manifold $N$ of dimension 21 on which there are two real analytic 1-forms $\omega$ and $\theta$ with values in $V$ and $\mathfrak{g}_2$ respectively, and a real-analytic function $T : N \to \mathfrak{g}_2$ such that $\left( \omega, \theta \right)$ is a real analytic coframing on N, equations (\ref{eq:CFlatCartI}), (\ref{eq:CFlatdT}), and (\ref{eq:CFlatCartII}) are satisfied on $N$, and there exists a $u_0 \in N$ for which $T \left( u_0 \right) = T_0.$
	
	By (\ref{eq:CFlatCartI}), the equation $\omega = 0$ defines an integrable plane field of codimension 7 on $N$. After shrinking $N$ to an open neighbourhood of $u_0$ if necessary, an application of the Frobenius Theorem shows there is a submersion $y : N \to \R^7$ with $y \left( u_0 \right) = 0$ so that the leaves of this plane field are the fibres of $y$ and, moreover, that $d y = p \omega$ for some function $p : N \to GL(7,\R)$ that satisfies $p \left( u_0 \right) = I_7.$
	
	By (\ref{eq:CFlatCartI}), the 3-form
	\begin{align}\label{eq:CFlatUpsphi}
	\upvarphi = \tfrac{1}{6} \eps_{ijk} \omega_i \wedge \omega_j \wedge \omega_k
	\end{align}
	is closed. Since $\upvarphi$ is $y$-semibasic, and since, by definition the fibres of $y$ are connected, it follows that $\upvarphi$ is the pullback to $N$ of a closed $\G_2$-structure on the open set $U = y \left( N \right).$ Denote this $\G_2$-structure by $\varphi$.
	
	Let $\pi : \mathcal{B} \to U$ be the $\G_2$-coframe bundle associated to $\varphi.$ Define a mapping $ \sigma : N \to \mathcal{B}$ as follows: If $y \left( u \right) = x \in U,$ then $dy_u : T_u N \to T_x \R^7 \cong \R^7$ is surjective, and by construction, has the same kernel as $\omega_u : T_u N \to \R^7.$ Thus, there is a unique linear isomorphism $\sigma \left( u \right) : T_x \R^7 \to \R^7$ so that $\omega_u = \sigma \left( u \right) \circ d y_u.$ Using the standard identification $T_x \R^7 \cong \R^7,$ $\sigma \left( u \right)$ is $y \left( u \right)^{-1} \in \mathrm{GL}(7,\R).$ That $\sigma \left( u \right)$ is a $\G_2$-coframe follows from the form (\ref{eq:CFlatUpsphi}) of $\upvarphi.$ Since $\left( \omega, \theta \right)$ is a coframe, it follows that $\sigma : N \to \mathcal{B}$ is an open immersion. After possibly shrinking $N$ again, $N$ may be identified with an open subset of $\mathcal{B}.$
	
	The structure equations (\ref{eq:CFlatCartI}), (\ref{eq:CFlatdT}), and (\ref{eq:CFlatCartII}) now become identified with the structure equations of $\mathcal{B},$ implying the $\G_2$-structure $\varphi$ is closed with conformally flat metric and that the torsion function $T$ takes the value $T_0$ at $u_0 \in \mathcal{B}.$ This proves the existence statement.
	
	Uniqueness in the real-analytic category follows immediately from Theorem 2 of \cite{BryEDSNotes}. In fact, the theorem can be proved using only ODE techniques, and uniqueness is given as long as $\varphi$ is sufficiently differentiable for $\mathcal{B}$ to exist as a differentiable bundle and $T$ to exist and be differentiable. It suffices that $\varphi$ be $C^2.$ The existence theorem provides a real analytic example, so uniqueness then proves $\varphi$ is real-analytic.
\end{proof}

The function $T: \mathcal{B} \to \mathfrak{g}_2$ is $\G_2$-equivariant, so the map $\left\lbrack T \right\rbrack$ that sends a point $p \in M$ to the adjoint orbit $\G_2 \cdot T(u),$ where $u$ is any element of the fibre $\pi^{-1} (p),$ is well-defined. From the discussion in \S\ref{sssect:V0qreps}, every element of $\mathfrak{g}_2$ is conjugate to a unique element of the form
\begin{equation}\label{eq:adgG2elt}
\mu_1 E_{23} + \mu_2 E_{45} - \left( \mu_1 + \mu_2 \right) E_{67}, \:\:\: 0 \leq \mu_1 \leq \mu_2,
\end{equation}
where $E_{ij}$ denotes the matrix with $(i,j)$ entry equal to $1,$ $(j,i)$ entry equal to $-1,$ and all other entries equal to $0.$ Let $A$ denote the set of elements of the form (\ref{eq:adgG2elt}). The function $\left\lbrack T \right\rbrack$ may be then thought of as a map $M \to A.$ It follows from Theorem \ref{thm:CFlatExist} that two germs of closed $\G_2$-structures with conformally flat induced metric at $p$ and $q$  are isomorphic if and only if $\left\lbrack T \right\rbrack (p) = \left\lbrack T \right\rbrack (q).$

\begin{cor}
	The set of isomorphism classes of germs of closed $\G_2$-structures with conformally flat induced metric is in bijection with $A.$
\end{cor}

\begin{thm}\label{thm:Class}
	A closed $\G_2$-structure $\varphi$ with (locally) conformally flat induced metric $g_{\varphi}$ is, up to constant rescaling, locally equivalent to one of the three examples $\R^7, \R^7 \setminus \lbrace 0 \rbrace,$ and $\Lambda^2_+ \mathbb{CP}^2$ presented in \S\ref{ssect:Egs}.
\end{thm}

\begin{proof}
	Let $N_1 = \R^7,$ $N_2 = \R^7 \setminus \lbrace 0 \rbrace$, and $N_3 = \Lambda^2_+ \mathbb{CP}^2$ be equipped with the $\G_2$-structures described in \S\ref{ssect:Egs}. From the discussion after Theorem \ref{thm:CFlatExist}, it is enough to show that for every element $X \in A,$ there is some constant $\lambda \in \R_+$ so that $\lambda X$ is in the image of $\left\lbrack T \right\rbrack$ applied to one of the $N_i.$
	
	For example \ref{sssect:EgFlat}, the torsion function vanishes identically, so
	\begin{equation}
	\left\lbrack T \right\rbrack \left( N_1 \right) = \left\lbrace 0 \right\rbrace.
	\end{equation}
	
	For example \ref{sssect:EgNonFlatR7}, a $\G_2$-coframe $\left( \omega_1, \ldots, \omega_7\right)$ is furnished by
	\begin{equation}
	\omega_1 = dt, \:\: \omega_{1+i} = \tfrac{2}{3} t \alpha_i, i=1,2 \:\:\:, \omega_{3+j} = \tfrac{16}{t^2} \theta_j, j=1, \ldots, 4,
	\end{equation}
	and the torsion 2-form $\tau$ is given in terms of this coframe by
	\begin{equation}
	\tau = \tfrac{4}{t} \left( 2 \omega_2 \wedge \omega_3 - \omega_4 \wedge \omega_5 - \omega_6 \wedge \omega_7 \right).
	\end{equation}
	Thus,
	\begin{equation}
	\left\lbrack T \right\rbrack \left( N_2 \right) = \left\lbrace  x E_{23} + x E_{45} - 2 x E_{67} \mid x \in \R_{+} \right\rbrace
	\end{equation}
	
	For example \ref{sssect:Lam2CP2}, the coframe $\alpha_1, \ldots, \alpha_7$ is a $\G_2$-coframe, and the torsion 2-form $\tau$ is given in terms of this coframe by
	\begin{equation}
	\tau = \frac{4 \sinh(r)}{3 \cosh(r)^{\frac{1}{3}}} \alpha_{23} + \frac{6 \cosh(r) - 2 \sinh(r)}{3 \cosh(r)^{\frac{1}{3}}} \alpha_{45} - \frac{6 \cosh(r) + 2 \sinh(r)}{3 \cosh(r)^{\frac{1}{3}}} \alpha_{67},
	\end{equation}
	Thus,
	\begin{equation}
	\left\lbrack T \right\rbrack \left( N_3 \right) = \left\lbrace  x E_{23} + y E_{45} -\left(x+ y \right) E_{67} \vline\ \: \begin{aligned}
	& x, y \in \R_+, \\
	& \!\!\! \left(x-y\right) \left( x+2y \right) \left(2 x + y \right) = -16
	\end{aligned}  \right\rbrace.
	\end{equation}
	
	It follows that
	\begin{equation}
	\R_+ \cdot \left\lbrack T \right\rbrack \left( N_1 \cup N_2 \cup N_3 \right) = A,
	\end{equation}
	completing the proof.
\end{proof}

Although the statement of Theorem \ref{thm:Class} is local, the strong rigidity imposed by the result makes it easy to give global consequences.

\begin{cor}\label{cor:CompleteCFlt}
	Let $(M, \varphi)$ be a 7-manifold endowed with a closed $\G_2$-structure such that $g_\varphi$ is conformally flat and complete. Then $\varphi$ is locally equivalent to the flat $\G_2$-structure $\phi$ on $V = \R^7$ and $M$ is the quotient of $V$ by a discrete group of $\G_2$-automorphisms.
\end{cor}

\begin{cor}\label{cor:H7noG2}
	There is no closed $\G_2$-structure inducing a metric of constant negative sectional curvature.
\end{cor}

%%%%%%%%%%%%%%%%%%%%%%%%%%%%%%%%%%%%%%%%%%%%%%%%%%%%%%%%%%
\section{A related family of closed $\G_2$-structures}\label{sect:Lam25}
%%%%%%%%%%%%%%%%%%%%%%%%%%%%%%%%%%%%%%%%%%%%%%%%%%%%%%%%%%

A key step in the classification \S\ref{sect:ConfClosedG2}  of closed $\G_2$-structures with conformally flat metric is the interesting coincidence that the vanishing of the Weyl curvature tensor of $g_{\varphi}$ is equivalent to equations (\ref{eqs:HCSConf}), which prescribe all of the second-order invariants $H, C,$ and $S$ of $\varphi$ as quadratic functions in the components of the first-order invariant $T.$

Motivated by this point, it is a reasonable question to ask if there are other closed $\G_2$-structures $\varphi$ for which the second-order invariants $H, C,$ and $S$ of $\varphi$ are given by quadratic functions in the components of the first-order invariant $T.$ Recall that $H, C,$ and $S$ are $\G_2$-equivariant functions on the $\G_2$-coframe bundle $\mathcal{B}$ taking values in $\V_{2,0}, \V_{1,1},$ and $\V_{0,2}$ respectively, and $T$ is a $\G_2$-equivariant function on $\mathcal{B}$ taking values in $\mathfrak{g}_2.$ 

There is an irreducible decomposition
\begin{equation}
\mathrm{Sym}^2 \left( \mathfrak{g}_2 \right) = \V_{0,2} \oplus \V_{2,0} \oplus \V_{0,0},
\end{equation}
and it follows from Schur's Lemma that there is a 2-parameter family of ways in which $H, C,$ and $S$ may be prescribed as a $\G_2$-equivariant quadratic function of the components of $T.$ Namely, for $\left( \lambda, \mu \right) \in \R^2,$ there are the equations
\begin{subequations}\label{eqs:Lam25HCS}
	\begin{align}
	& H_{ij} = \frac{1-6 \lambda}{7} \left( T_{ik}T_{kj} + \frac{1}{7} \delta_{ij} T_{pq}T_{pq} \right), \label{eq:QFroH} \\
	&C_{ijk} = 0, \label{eq:QFroC} \\
	&S_{ijkl} = \mu \left( 3 T_{ij}T_{kl} +\tfrac{3}{16} \left( 2\eps_{pij}\eps_{qkl} - 3 \eps_{pik}\eps_{qjl} + 3 \eps_{pil}\eps_{qjk} \right) T_{pm}T_{mq} \right. \label{eq:QFroS} \\
	& + \tfrac{9}{16} \left( \delta_{ik} T_{jm}T_{ml} - \delta_{il} T_{jm}T_{mk} - \delta_{jk} T_{im}T_{ml} + \delta_{jl} T_{im}T_{mk} \right) \nonumber \\
	& \left. + \tfrac{1}{16}  \left( 4 \eps_{ijkl} + \delta_{ik} \delta_{jl} - \delta_{il} \delta_{jk} \right) T_{pq}T_{pq} \right). \nonumber
	\end{align}
\end{subequations}
Note that equation (\ref{eq:QFroH}) is equivalent to the $\lambda$-quadratic equation (\ref{eq:LamQuad}).

Substituting equation (\ref{eqs:Lam25HCS}) into the structure equations (\ref{eq:CartanIG2}), (\ref{eq:dTcomp}), and (\ref{eq:dthetcomp}) and differentiating, the equations $d^2 T = 0$ and $d^2 \theta = 0$ are satisfied if and only if $(\lambda, \mu)$ is equal to $\left(-1/8, 1 \right)$ or $\left(2/5, -2/25 \right).$ The former case corresponds to the equations for conformal flatness (\ref{eqs:HCSConf}), while the latter case will be studied in the remainder of this section.

Substitution of equations (\ref{eqs:Lam25HCS}) with $(\lambda, \mu) = \left(2/5, -2/25 \right)$ into the structure equations (\ref{eq:CartanIG2}), (\ref{eq:dTcomp}), and (\ref{eq:dthetcomp}) gives the structure equations
\begin{subequations}\label{eqs:Lam25Struct}
	\begin{align}
	d & \omega_i = -\theta_{ij} \wedge \omega_j + \eps_{ijk} T_{kl} \omega_j \wedge \omega_l, \label{eq:Lam25CartI} \\
	d & T_{jk} = -\tfrac{1}{5} \left( \eps_{jkl} T_{lm}T_{mi} + 3 \eps_{kil} T_{lm}T_{mj} + 3 \eps_{ijl}T_{lm}T_{mk} + \eps_{ijk}T_{pq}T_{pq} \right)\omega_i \label{eq:Lam25dT} \\
	& + T_{jl}\theta_{lk}-T_{kl}\theta_{lj}, \nonumber \\	
	d & \theta_{ij} = - \theta_{ik} \wedge \theta_{kj} + \tfrac{1}{25} \left( -{6} T_{ij}T_{kl} \right. \label{eq:Lam25CartII} \\
	&\left. - \left( {2} \eps_{pij}\eps_{qkl} - {3} \eps_{pik}\eps_{qjl} + {3} \eps_{pil}\eps_{qjk} \right) T_{pm}T_{mq} \right. \nonumber  \\
	& - 13 \left( \delta_{ik} T_{jm}T_{ml} - \delta_{il} T_{jm}T_{mk} - \delta_{jk} T_{im}T_{ml} + \delta_{jl} T_{im}T_{mk} \right) \nonumber \\
	&+ 5 \left( \eps_{ijkq} T_{lm}T_{mq} - \eps_{lijq}T_{km}T_{mq} \right) \nonumber \\
	& \left. + \left( 2 \eps_{ijkl} - 7 \delta_{ik} \delta_{jl} + 7 \delta_{il} \delta_{jk} \right) T_{pq}T_{pq} \right) \omega_k \wedge \omega_l. \nonumber
	\end{align}
\end{subequations}
The derivatives of equations (\ref{eqs:Lam25Struct}) are identities.

By the discussion above, a closed $\G_2$-structure satisfying equations (\ref{eqs:Lam25HCS}) with $(\lambda, \mu)$ $= \left(2/5, -2/25 \right)$ is $2/5$-quadratic, and this rules out the existence of compact examples with non-zero torsion. The following proposition is proven in the exact same manner as Proposition \ref{prop:CpctCFlat}.

\begin{prop}\label{prop:CpctLam25}
	If $(M, \varphi)$ is a compact 7-manifold endowed with a closed $\G_2$-structure satisfying equations (\ref{eqs:Lam25HCS}) with $(\lambda, \mu) =$ $\left(2/5, -2/25 \right)$, then the metric $g_{\varphi}$ is flat and the $\G_2$-structure is locally equivalent to the standard $\G_2$-structure $\phi$ on $V = \R^7.$
\end{prop}

\subsection{Examples}\label{ssect:Lam25Egs}

This section contains three examples of closed $\G_2$-structures satisfying equations (\ref{eqs:Lam25HCS}) with $(\lambda, \mu) = \left(2/5, -2/25 \right).$

\subsubsection{The flat $\G_2$-structure on $\R^7.$}

The standard 3-form $\phi$ on $V=\R^7$ has $T=H=C=S=0,$ so equations (\ref{eqs:Lam25HCS}) are satisfied for all values of $\lambda$ and $\mu.$ 

It is clear from the form of equations (\ref{eqs:Lam25Struct}) that any torsion-free $\G_2$-structure satisfying equations (\ref{eqs:Lam25HCS}) with $(\lambda, \mu) = \left(2/5, -2/25 \right)$ must be locally equivalent to $\phi.$

\subsubsection{An example on $\R_{+} \times \R^6$}

Let $\mathrm{H}$ denote the 6-dimensional real Lie group underlying the complex 3-dimensional Heisenberg group. The complex Heisenberg group may be globally parametrised as
\begin{equation}
\mathbb{C}^2 \oplus \mathbb{C} \ni \left( x, z \right) \leftrightarrow \left( \begin{array}{ccc}
1 & x_1 & z + \tfrac{1}{2}x_1x_2 \\
0 & 1 & x_2 \\
0 & 0 & 1
\end{array}\right) \in \mathrm{H}.
\end{equation}
The group $\mathrm{U}(2)$ acts by automorphisms on $\mathrm{H}$ as $A \cdot \left( x, z \right) = \left( A x, \det \! A \: z \right).$ Let $\mathrm{G}$ denote the semi-direct product $\mathrm{U}(2) \ltimes \mathrm{H}$ with respect to this action. The left-invariant Maurer-Cartan forms of $\mathrm{U}(2)$ and $\mathrm{H}$ may be written as
\begin{equation}
\left( \begin{array}{cc}
-i \rho - i \xi_1 & \xi_2 + i \xi_3 \\
-\xi_2 + i \xi_3 & -i \rho + i \xi_1
\end{array}\right) \:\:\: \text{and} \:\:\: \left( \begin{array}{ccc}
0 & \eta_1 + i \eta_2 & -\tfrac{5}{3} \left(\alpha_2 + i \alpha_1 \right) \\
0 & 0 & \eta_3 + i \eta_4 \\
0 & 0 & 0
\end{array}\right)
\end{equation}
respectively. The structure equations of $\mathrm{G}$ are

%Let $\mathbb{C}^2 \cong \mathbb{H}$ be endowed with the standard hyper-K\"ahler structure $( \Omega_{I}, \Omega_{J}, \Omega_{K} ),$ and consider the $\mathrm{T}^2$-bundle $X \to \mathbb{C}^2$ consisting of the fibre product of the $\mathrm{S}^1$-bundles with curvature forms $\Omega_{J}$ and $\Omega_K.$
%
%Let $\mathcal{Q}$ be the $\mathrm{U}(2)$-coframe bundle of $\mathbb{C}^2$ with respect to the K\"ahler structure given by $\Omega_I,$ and denote the $\mathcal{C}^2$-valued tautological form on $\mathcal{Q}$ by $(\eta_1+ i\eta_2, \eta_3 +i \eta_4),$ and the $\mathfrak{u}(2)$-valued Levi-Civita form by
%\begin{equation}
%\left( \begin{array}{cc}
%i \rho + i \xi_1 & \xi_2 + i \xi_3 \\
%-\xi_2 + i \xi_3 & i \rho - i \xi_1
%\end{array}\right).
%\end{equation}
% Then, on the total space $\mathcal{X}$ of the pullback of the bundle $X$ to $\mathcal{Q},$ there is a coframe $( \eta_1, \eta_2, \eta_3, \eta_4, \alpha_1, \alpha_2, \xi_1, xi_2, \xi_3, \rho_1)$ satisfying the structure equations
 \begin{subequations}\label{eqs:Eg2Lam25}
 	\begin{align}
 	d \left(\begin{array}{c}
 	\eta_1 \\
 	\eta_2 \\
 	\eta_3 \\
 	\eta_4
 	\end{array}\right) &= - \left( \begin{array}{cccc}
 	0 & \xi_1 + \rho_1 & \xi_2 & \xi_3 \\
 	-\xi_1 - \rho_1 & 0 & -\xi_3 & \xi_2 \\
 	- \xi_2 & \xi_3 & 0 & -\xi_1 + \rho_1 \\
 	-\xi_3 & -\xi_2 & \xi_1 - \rho_1 & 0
 	\end{array}\right) \wedge \left(\begin{array}{c}
 	\eta_1 \\
 	\eta_2 \\
 	\eta_3 \\
 	\eta_4
 	\end{array}\right) , \\
 	d \left( \begin{array}{c}
 	\alpha_1 \\
 	\alpha_2
 	\end{array} \right) &= - \left( \begin{array}{cc}
 	0 & -2\rho_1 \\
 	2 \rho_1 & 0 
 	\end{array}\right) \wedge \left( \begin{array}{c}
 	\alpha_1 \\
 	\alpha_2
 	\end{array} \right) + \frac{3}{5} \left(\begin{array}{c}
 	\eta_1 \wedge \eta_4 + \eta_2 \wedge \eta_3 \\
 	\eta_1 \wedge \eta_3 - \eta_2 \wedge \eta_4
 	\end{array}\right),
 	\end{align}
 \end{subequations}
and $d \rho_1 = 0, \:\: d \xi_i = 2 \xi_j \wedge \xi_k$ for $\left(i,j,k \right)$ a cyclic permutation of $( 1,2,3).$
%
%%CHECK!!
%
On the manifold $\R_+ \times \mathrm{G},$ with coordinate $t$ in the $\R_+$-direction, define a $\mathrm{G}$-invariant 3-form $\tilde{\varphi}$ by
\begin{equation}
\begin{aligned}
\tilde{\varphi} =& \tfrac{625}{t^4} d t \wedge \alpha_1 \wedge \alpha_2 + \tfrac{5}{t} d t \wedge \left( \eta_1 \wedge \eta_2 + \eta_3 \wedge \eta_4 \right) \\
&+ \tfrac{125}{t^3} \left( \alpha_1 \wedge \left( \eta_1 \wedge \eta_3 - \eta_2 \wedge \eta_4 \right) + \alpha_2 \wedge \left( -\eta_1 \wedge \eta_4 - \eta_2 \wedge \eta_3 \right) \right).
\end{aligned}
\end{equation}
The structure equations (\ref{eqs:Eg2Lam25}) imply that $d \tilde{\varphi} = 0.$ In fact, $\tilde{\varphi}$ is the pullback to $\R_+ \times \mathrm{G}$ of a closed $\G_2$-structure on $\R_+ \times \mathrm{G}/\mathrm{U}(2).$ Calculation shows that $\varphi$ satisfies equations (\ref{eqs:Lam25HCS}) with $(\lambda, \mu) = \left(2/5, -2/25 \right).$ The torsion 2-form is given by
\begin{equation}
\tau = \tfrac{25}{t^2} \left( \eta_1 \wedge \eta_2 + \eta_3 \wedge \eta_4 \right) - \tfrac{6250}{t^5} \alpha_1 \wedge \alpha_2.
\end{equation}

The scalar curvature of the induced metric $g_{\varphi}$ is
\begin{equation}
\mathrm{Scal}( g_{\varphi}) = -\frac{75}{t^2},
\end{equation}
and it follows that the metric $g_{\varphi}$ is incomplete.

The space $\mathrm{G}/\mathrm{U}(2)$ is diffeomorphic to $\R^6,$ so $\varphi$ defines a closed $\G_2$-structure on $\R_{+} \times \R^6.$

\subsubsection{An example on $\Lambda^2_{+} \mathbb{CH}^2$}

Write the left-invariant Maurer-Cartan form of $\SU(2,1)$ as
\begin{equation}
\mu = \left( \begin{array}{ccc}
\tfrac{2}{3} i \xi & \eta_4 + i \eta_5 & -\eta_7 - i \eta_6 \\
\eta_4 - i \eta_5 & i \zeta - \tfrac{1}{3} i \xi & \eta_2 + i \eta_3 \\
\eta_7 - i \eta_6 & \eta_2 - i \eta_3 & - i \zeta - \tfrac{1}{3} i \xi
\end{array} \right).
\end{equation}

On the manifold $\left( \R_+ \setminus \left\lbrace \tfrac{1}{2} \log 2 \right\rbrace \right) \times \SU(2,1)$ with coordinate $r$ in the $\R_+$-direction, define 1-forms $\alpha_1, \ldots, \alpha_7$ by
\begin{equation}
\begin{aligned}
\alpha_1 &= \frac{5 \: d r}{\sinh(r)^{\frac{2}{3}}} , & \left( \begin{array}{c}
\alpha_2 \\
\alpha_3
\end{array} \right) &= \frac{10 \cosh(r)} {\sinh(r)^{\frac{2}{3}}} \left( \begin{array}{c}
\eta_2 \\
\eta_3
\end{array}\right), \\
\left( \begin{array}{c}
\alpha_4 \\
\alpha_5
\end{array} \right) & = \frac{5\sqrt{2} \left( 2 e^{\tfrac{1}{2}t} - e^{-\tfrac{3}{2}t} \right) }{\sinh(r)^{\frac{1}{6}}} \left( \begin{array}{c}
\eta_4 \\
\eta_5
\end{array}\right), \\
\left( \begin{array}{c}
\alpha_6 \\
\alpha_7
\end{array} \right) & = \frac{5\sqrt{2} \left( 2 e^{-\tfrac{1}{2}t} - e^{\tfrac{3}{2}t} \right) }{\sinh(r)^{\frac{1}{6}}} \left( \begin{array}{c}
\eta_6 \\
\eta_7
\end{array}\right),
\end{aligned}
\end{equation}
and a 3-form $\tilde{\varphi}$ by
\begin{equation}
\tilde{\varphi} =  \alpha_{123} + \alpha_{145} + \alpha_{167} + \alpha_{246} - \alpha_{257} - \alpha_{347} - \alpha_{356},
\end{equation}
where $\alpha_{ijk}$ stands for the wedge product $\alpha_i \wedge \alpha_j \wedge \alpha_k.$ The Maurer-Cartan equation $d \mu = - \mu \wedge \mu$ implies that $\tilde{\varphi}$ satisfies $d \tilde{\varphi} = 0.$ In fact, $\tilde{\varphi}$ is the pullback to $\left( \R_+ \setminus \left\lbrace \tfrac{1}{2} \log 2 \right\rbrace \right) \times \SU(2,1)$ of a closed $\G_2$-structure on the space $\left( \R_+ \setminus \left\lbrace \tfrac{1}{2} \log 2 \right\rbrace \right) \times \times \SU(2,1) / \mathrm{T}^2.$ Similarly to example in \S\ref{sssect:Lam2CP2}, $\varphi$ extends over $\left\lbrace r = \tfrac{1}{2} \log 2 \right\rbrace$ to define a closed $\G_2$-structure on $\Lambda^2_+ \mathbb{CH}^2$ that is cohomogeneity-one under the action of $\mathrm{SU}(2,1).$ Computation using the Maurer-Cartan equation shows that $\varphi$ satisfies equations (\ref{eqs:Lam25HCS}) with $(\lambda, \mu) = \left(2/5, -2/25 \right).$

Letting $\pi: \mathrm{SU}(2,1) \to \mathrm{SU}(2,1)/\mathrm{T}^2$ denote the coset projection, the torsion 2-form $\tau$ of $\varphi$ satisfies
\begin{equation}
\tau = -\frac{2 \cosh(r)}{3 \sinh(r)^{\frac{1}{3}}} \alpha_{23} + \frac{3 \sinh(r) + \cosh(r)}{3 \sinh(r)^{\frac{1}{3}}} \alpha_{45} + \frac{3 \sinh(r) - \cosh(r)}{3 \sinh(r)^{\frac{1}{3}}} \alpha_{67},
\end{equation}
and the scalar curvature of $g_{\varphi}$ may be computed using formula (\ref{eq:RiemCurvDecompScal}) to be
\begin{equation}
\mathrm{Scal}(g_{\varphi}) = \frac{3 - 4 \cosh(r)^2}{3 \sinh(r)^{\frac{2}{3}}} .
\end{equation}
Note that the scalar curvature blows up as $r \to \infty,$ and $r \to 0^{+}$ and these are at finite distance, since
\begin{equation}
\int_{r=0}^{r=\infty} \alpha_1 = \int_{r=0}^{r=\infty} \frac{5 \: d r}{\sinh(r)^{\frac{2}{3}}} < \infty.
\end{equation}
It follows that the metric $g_{\varphi}$ is incomplete.

\subsection{Classification}

The classification of closed $\G_2$-structures satisfying equations (\ref{eqs:Lam25HCS}) with $(\lambda, \mu) = \left(2/5, -2/25 \right)$ proceeds in exactly the same manner as the classification of closed $\G_2$-structures with conformally flat induced metric, and the following theorems are proven using the same techniques as are used in the proofs of the corresponding theorems in \S\ref{ssect:ConfClass}.

\begin{thm}\label{thm:Lam25Exist}
	For any $T_0 \in \mathfrak{g}_2,$ there exists a closed $\G_2$-structure $\varphi$ satisfying equations (\ref{eqs:Lam25HCS}) with $(\lambda, \mu) = \left(2/5, -2/25 \right)$ on a neighbourhood $U$ of $0 \in \R^7$ whose $\G_2$-coframe bundle $\pi : \mathcal{B} \to U$ contains a $u_0 \in \mathcal{B}_0 = \pi^{-1} \left( 0 \right)$ for which $T \left( u_0 \right) = T_0.$ Any two real analytic closed $\G_2$-structures satisfying equations (\ref{eqs:Lam25HCS}) with $(\lambda, \mu) = \left(2/5, -2/25 \right)$ that satisfy this property are isomorphic on a neighbourhood of $0 \in \R^7.$ Furthermore, any closed $\G_2$-structure satisfying equations (\ref{eqs:Lam25HCS}) with $(\lambda, \mu) = \left(2/5, -2/25 \right)$ that is $C^2$ is real-analytic.
\end{thm}

\begin{thm}\label{thm:Lam25Class}
	A closed $\G_2$-structure $\varphi$ satisfying equations (\ref{eqs:Lam25HCS}) with $(\lambda, \mu) = \left(2/5, -2/25 \right)$ is, up to constant rescaling, locally equivalent to one of the three examples $\R^7, \R_+ \times \R^6,$ and $\Lambda^2_+ \mathbb{CH}^2$ presented in \S\ref{ssect:Lam25Egs}.
\end{thm}

\bibliography{ConfClosedG2Refs}

\end{document}